\documentclass[reqno]{amsart}

\usepackage{amsmath}
\usepackage{amsfonts}
\usepackage{amssymb}
\usepackage{amscd}
\usepackage{stmaryrd}
\usepackage[matrix,arrow,cmtip,rotate,curve,arc]{xy}
\SelectTips{cm}{}
\input xy
\xyoption{all}
\usepackage{graphics}
\usepackage{graphicx}
\usepackage{fancyhdr}
\usepackage{ushort}

\DeclareSymbolFont{bbold}{U}{bbold}{m}{n}
\DeclareMathSymbol{\bbzero}{\mathord}{bbold}{`0}
\DeclareMathSymbol{\bbone}{\mathord}{bbold}{`1}

\newcommand{\bbI}{\mathbb I}

\newcommand{\bbL}{\mathbb L}

\newcommand{\bbZ}{\mathbb Z}

\newcommand{\mcC}{\mathcal C}
\newcommand{\mcD}{\mathcal D}

\newcommand{\mcF}{\mathcal F}

\newcommand{\mcI}{\mathcal I}
\newcommand{\mcJ}{\mathcal J}

\newcommand{\mcM}{\mathcal M}
\newcommand{\mcN}{\mathcal N}

\newcommand{\mcP}{\mathcal P}
\newcommand{\mcQ}{\mathcal Q}

\newcommand{\mcV}{\mathcal V}

\newcommand{\sfB}{\mathsf B}

\newcommand{\sfI}{\mathsf I}

\newcommand{\Ab}{\mathsf{Ab}}

\newcommand{\cosimp}{\mathsf{c}}
\newcommand{\Cat}{\mathsf{Cat}}

\newcommand{\Ch}{\mathsf{Ch}}

\newcommand{\ob}{\mathop\mathsf{ob}\nolimits}

\newcommand{\simp}{\mathsf{s}}
\newcommand{\Set}{\mathsf{Set}}
\newcommand{\sAb}{\simp\Ab}

\newcommand{\sSet}{\simp\Set}

\newcommand{\Ho}{\mathop\mathrm{Ho}\nolimits}
\renewcommand{\hom}{\mathop\mathsf{hom}\nolimits}

\newcommand{\op}{\mathrm{op}}

\newcommand{\Ra}{\Rightarrow}

\newcommand{\xra}[1]{\xrightarrow{#1}}

\newcommand{\ra}{\rightarrow}

\newcommand{\da}{\downarrow}

\newcommand{\hra}{\hookrightarrow}
\newcommand{\xlra}[1]{\xrightarrow{\ #1\ }}
\newcommand{\xlla}[1]{\xleftarrow{\ #1\ }}
\newcommand{\lra}{\longrightarrow}
\newcommand{\lla}{\longleftarrow}

\newdir{c}{{}*!/-5pt/@^{(}}
\newdir{d}{{}*!/-5pt/@_{(}}
\newdir{ >}{{}*!/-5pt/@{>}}
\newdir{s}{{}*!/+10pt/@{}}
\newdir{|>}{{}*!/2pt/@{|}*@{>}}

\newcommand{\embed}[1]{\xymatrix@1{{}\ar@{c->}[r]^{#1} & {}}}

\newcommand{\colim}{\mathop\mathrm{colim}}

\newcommand{\const}{\mathop\mathrm{cst}\nolimits}

\newcommand{\fib}{\mathop\mathrm{fib}\nolimits}

\newcommand{\hocolim}{\mathop\mathrm{hocolim}}

\newcommand{\id}{\mathrm{id}}
\newcommand{\Id}{\mathrm{Id}}

\newcommand{\pcm}{\mathop\mathrm{pcm}\nolimits}

\newlength{\hlp}

\entrymodifiers={+!!<0pt,\the\fontdimen22\textfont2>}

\theoremstyle{plain}
\newtheorem{theorem}{Theorem}
\newtheorem{lemma}[theorem]{Lemma}

\newtheorem{proposition}[theorem]{Proposition}

\newtheorem*{scorollary}{Corollary}
\newtheorem*{observation}{Observation}

\theoremstyle{definition}

\newtheorem*{sexample}{Example}

\newtheorem*{sdefinition}{Definition}

\theoremstyle{remark}
\newtheorem*{remark}{Remark}
\newtheorem*{remarks}{Remarks}

\renewcommand{\rightarrowtail}{\xymatrix@1@=15pt{\ar@{ >->}[r]&}}
\newcommand{\cof}{\mathrm{cof}}
\renewcommand{\fib}{\mathrm{fib}}

\newcommand{\col}{\;\!{:{}}}
\newcommand{\loc}{{{}:}\;\!}

\begin{document}

\title{Homotopy weighted colimits}
\author{Luk\'a\v{s} Vok\v{r}\'inek}
\email{koren@math.muni.cz}
\address{Department of Mathematics and Statistics\\Masaryk University\\Kotl\'a\v rsk\'a 2\\611 37 Brno\\Czech Republic}
\date{\today}
\thanks{
The research was supported by the project CZ.1.07/2.3.00/20.0003 of the Czech Ministry of Education and by the grant P201/11/0528 of the Czech Science Foundation (GA \v CR)}

\begin{abstract}
Let $\mcV$ be a cofibrantly generated closed symmetric monoidal model category and $\mcM$ a model $\mcV$-category. We say that a weighted colimit $W\star_\mcC D$ of a diagram $D$ weighted by $W$ is a homotopy weighted colimit if the diagram $D$ is pointwise cofibrant and the weight $W$ is cofibrant in the projective model structure on $[\mcC^\op,\mcV]$. We then proceed to describe such homotopy weighted colimits through homotopy tensors and ordinary (conical) homotopy colimits. This is a homotopy version of the well known isomorphism $W\star_\mcC D\cong\smallint\nolimits^\mcC(W\otimes D)$. After proving this homotopy decomposition in general we study in some detail a few special cases. For simplicial sets tensors may be replaced up to weak equivalence by conical homotopy colimits and thus the weighted homotopy colimits have no added value. The situation is completely different for model dg-categories where the desuspension cannot be constructed from conical homotopy colimits. In the last section we characterize those $\mcV$-functors inducing a Quillen equivalence on the enriched presheaf categories.
\end{abstract}

\maketitle

\section*{Introduction}

This paper is concerned with the study of homotopy colimits and their enriched versions. These are defined in a model $\mcV$-category, i.e.~a model category enriched in a closed symmetric monoidal model category $\mcV$. Instead of deriving the weighted colimit functor we simply describe cases in which the weighted colimit is homotopy invariant and call it then the homotopy weighted colimit. Ordinary homotopy colimits in simplicial model categories may be defined as particular (homotopy) weighted colimits via the so-called Bousfield-Kan formula. We first prove its generalization to an arbitrary base $\mcV$.

In the opposite direction we express homotopy weighted colimits via ordinary homotopy colimits and (again homotopy invariant cases of) tensors. This should be compared with the well known non-homotopy version $W\star_\mcC D\cong\smallint\nolimits^\mcC(W\otimes D)$ where the coend is essentially a (reflexive) coequalizer and a further tensor.

In the second section we study a few examples of enrichments in which we compare the expressive power of homotopy weighted colimits with that of ordinary homotopy colimits. Simplicial sets provide a basic example -- they act in an ``up to homotopy'' way on every model category and they also form a ``free model category on one generator''. In this respect one should expect that the enrichment in simplicial sets does not give anything in addition. From the non-homotopy point of view the weighted colimits over simplicial sets are generated by conical colimits and tensors with standard simplices. The same holds homotopically except tensors with standard simplices are now trivial. As a result one might compute homotopy weighted colimits by ordinary homotopy colimits although the diagram needs to be replaced by a weakly equivalent one. Otherwise there might be too few morphisms.

Our second example are abelian groups. In the non-homotopy world $\Ab$-weighted colimits can be described via conical colimits. The same holds in the homotopy setup $\mcV=\sAb$ of simplicial abelian groups or equivalently the non-negatively graded chain complexes. The unbounded case happens to be more interesting -- there are homotopy weighted colimits (namely desuspension, i.e.~the tensor with $\bbZ[-1]$) which cannot be recovered from ordinary homotopy colimits. Over $\Cat$ with its categorical model structure the tensor with the arrow category $[1]$ is also not expressible by ordinary homotopy colimits, i.e.~pseudocolimits.

In the third section we prove a version of a theorem of Dwyer and Kan about $\mcV$-functors $F\col\mcD\ra\mcC$ inducing Quillen equivalences on the presheaf categories
\[F_!\col[\mcD^\op,\mcV]\rightleftarrows[\mcC^\op,\mcV]\loc F^*.\]
Classically for $F$ essentially surjective this happens precisely when $F$ is homotopically full and faithful. In general we characterize such functors via homotopy absolute homotopy weighted colimits and describe these completely for $\mcV=\sSet$, $\sAb$ and $\Cat$.

\section{Homotopy weighted colimits}

\subsection{Homotopy colimits}
In any model category $\mcM_0$ it is possible to define a homotopy invariant version of a colimit called a homotopy colimit. There are many definitions, we will follow the approach of \cite{Hirschhorn}. Let $D\col\mcI\ra\mcM_0$ be a pointwise cofibrant diagram, i.e.~such that for all objects $i\in\mcI$ the value $Di\in\mcM_0$ is cofibrant. The understanding here is that we will not compute a homotopy colimit of a diagram which is not pointwise cofibrant and if we ever want to we first take its cofibrant replacement. There is a Reedy model structure on $\cosimp\mcM_0=[\Delta,\mcM_0]$, the category of cosimplicial objects in $\mcM_0$. Let $\const_\Delta D\col\mcI\ra\cosimp\mcM_0$ denote the diagram whose value $(\const_\Delta D)i$ is the constant cosimplicial object on $Di$. Let now $\widetilde D$ denote a pointwise Reedy cofibrant replacement of $\const_\Delta D$. We define
\[\hocolim\nolimits_\mcI D=\int^{(n,i)\in\Delta\times\mcI}N(i\da\mcI)_n\cdot(\widetilde{D}i)^n\]
where the $\cdot$ is the ``$\Set$-valued tensor''. This means that $S\cdot D$ denotes the coproduct of as many copies of $D$ as there are elements in $S$, i.e.~$S\cdot D=\coprod_{s\in S}D$. By definition
\[N(i\da\mcI)_n=\{i\ra i_0\ra\cdots\ra i_n\},\]
the set of all chains of $(n+1)$ morphisms in $\mcI$ starting from $i$. The above coend is in fact a $\Set$-enriched weighted colimit. As such we will write it as
\[\hocolim\nolimits_\mcI D=N(-\da\mcI)\cdot_{\Delta\times\mcI}\widetilde D\]
to distinguish it from the weighted colimits enriched in other monoidal categories that we will be using later.

Note that the homotopy colimit depends on the choice of the Reedy cofibrant replacement $\widetilde D$ but we will not reflect this in the notation. A different choice leads to a weakly equivalent (cofibrant) object. This suggests that one should not think of $\hocolim\nolimits_\mcI D$ as a specific object but rather as a weak equivalence class of objects. In this respect when speaking about homotopy colimits one should never ask the homotopy colimit to be isomorphic to an object but always weakly equivalent. We will demonstrate this for the presheaf categories. Any presheaf $X\in[\mcI^\op,\Set]$ is (isomorphic to) a colimit of representable functors. There is an analogous result for homotopy colimits. We will see in the proceeding that any model category has a ``homotopy enrichment'' in $\sSet$. Therefore the corresponding result should be concerned with ``simplicial presheaves'' $X\in[\mcI^\op,\sSet]$. Indeed \cite{Dugger} proved that every such $X$ is weakly equivalent to a homotopy colimit of representables (which are embedded in $[\mcI^\op,\sSet]$ via the canonical functor $\Set\ra\sSet$ sending a set to the corresponding constant=discrete simplicial set). We will generalize this result in allowing $\mcI$ to be a simplicial category. First we explain some generalities on enriched category theory.

\subsection{Weighted colimits}
Let $\mcM$ be an arbitrary category enriched over some closed symmetric monoidal category $\mcV_0$ and denote by $\mcM_0$ the underlying ordinary category of $\mcM$. In the proceeding we simply say that $\mcM$ is a $\mcV$-category. We will distinguish the $\mcV$-category $\mcV$ from its underlying ordinary category $\mcV_0$ in the same way we do for $\mcM$ and $\mcM_0$. Let $\mcC$ be a small $\mcV$-category and $W\col\mcC^\op\ra\mcV$ a $\mcV$-functor, a so-called weight. A $W$-weighted colimit of a diagram $D\col\mcC\ra\mcM$ is an object $W\star_\mcC D$ for which there is a natural isomorphism
\[[\mcC^\op,\mcV](W,\mcM(D-,X))\cong\mcM(W\star_\mcC D,X).\]
As a special case a tensor $K\otimes M$ is a weighted colimit indexed by the $\mcV$-category with one object and ``only the identity morphism'' and whose weight sends the unique object to $K$. Therefore the natural isomorphism becomes
\[\mcV(K,\mcM(M,X))\cong\mcM(K\otimes M,X).\]
It is classical that weighted colimits can be expressed through ordinary (non-enriched, conical) colimits and tensors if these exist in $\mcM$, namely
\begin{equation} \label{eqn_weighted_colimit_to_coend}
W\star_\mcC D\cong\int\nolimits^{c\in\mcC}Wc\otimes Dc.
\end{equation}
The coend can be further simplified to a conical colimit (while introducing further tensors) since it can be computed as a coequalizer
\begin{equation} \label{eqn_weighted_colimit_to_coequalizer}
\xymatrix{
\displaystyle\coprod_{c_0,c_1}Wc_1\otimes\mcC(c_0,c_1)\otimes Dc_0 \ar@<-2pt>[r] \ar@<2pt>[r] & \displaystyle\coprod\limits_c Wc\otimes Dc \ar[r] & W\star_\mcC D.
}
\end{equation}
Most results about weighted colimits we will be using can be found in \cite{Kelly}, the above descriptions appear both on page 53.

\subsection{Model $\mcV$-categories}
If both $\mcV$ and $\mcM$ are at the same time model categories (in a compatible way to be explained) one may define homotopy weighted colimits and ask if these can be computed using homotopy tensors and conical homotopy colimits. We start with explaining the compatibility between the monoidal and the model structure of $\mcV$ which is expressed in $\mcV_0$ being a (cofibrantly generated closed symmetric) \emph{monoidal model category}. We denote the monoidal product by $\otimes$ and the unit by $\sfI$. In the following we will need a cosimplicial resolution of $\sfI$ which we denote by $\bbI$. By definition it is a Reedy cofibrant replacement of the constant cosimplicial object $\const_\Delta\sfI$ with value $\sfI$ in $\cosimp\mcV_0=[\Delta,\mcV_0]$. Typically $\sfI$ is cofibrant in which case we may assume $\bbI^0=\sfI$. In general the first compatibility axiom says that the augmentation map $\bbI^0\ra\sfI$ induces for every cofibrant object $K$ a weak equivalence
\[\bbI^0\otimes K\xra{\sim}\sfI\otimes K\cong K.\]
The second compatibility axiom says that the tensor product $\mcV_0\times\mcV_0\ra\mcV_0$ is a left Quillen bifunctor. A good reference for monoidal model categories is \cite{Hovey}.

By a model $\mcV$-category we understand a $\mcV$-category $\mcM$ whose underlying ordinary category $\mcM_0$ is a model category and such that $\mcM$ has all tensors and cotensors and these satisfy an analogue of the compatibility conditions for $\mcV$, namely the tensor product $\mcV_0\times\mcM_0\ra\mcM_0$ should be a left Quillen bifunctor and for any cofibrant $M\in\mcM$ the augmentation map
\[\bbI^0\otimes M\xra{\sim}\sfI\otimes M\cong M\]
should be a weak equivalence. In particular $\mcV$ itself is a model $\mcV$-category.

With the cosimplicial resolution at hand one may readily define a couple of useful constructions. Firstly there is an (up to isomorphism unique) cocontinuous extension of $\bbI$,
\begin{equation} \label{eqn_embedding_of_sset}
\sSet_0\lra\mcV_0, \qquad K\longmapsto K\cdot_\Delta\bbI=\int\nolimits^{n\in\Delta}K_n\cdot\bbI^n.
\end{equation}
It is a left Quillen functor by Corollary~5.4.4 of \cite{Hovey} and endows any $\mcM_0$ with an ``up to homotopy'' simplicial module structure
\begin{equation} \label{eqn_simplicial_enrichment}
\sSet_0\times\mcM_0\lra\mcV_0\times\mcM_0\xlra{-\otimes-}\mcM_0
\end{equation}
(all the required axioms are only satisfied up to weak equivalences). Further, there is a geometric realization functor $|-|\col\simp\mcM\ra\mcM$ from the category of simplicial objects in $\mcM$ to $\mcM$. It is given simply as the weighted colimit
\begin{equation} \label{eqn_geometric_realization}
|X|=\bbI\star_\Delta X.
\end{equation}
As usual with geometric realizations it is homotopy invariant (i.e.~preserves weak equivalences) only on the subcategory of Reedy cofibrant diagrams. It is classical that restricting to the subcategory $\Delta_+^\op$ whose morphisms are the opposites of injective maps of ordinals one obtains a variant which is homotopy invariant on all pointwise cofibrant diagrams,
\[\hocolim\nolimits_{\Delta^\op}X\simeq\bbI\star_{\Delta_+}X.\]
We prefer to use the geometric realization for two reasons. Firstly it is shorter and secondly it does not matter two much since our diagrams will be Reedy cofibrant under almost exactly the same conditions under which they are pointwise cofibrant.

The cosimplicial resolution $\bbI$ might be used to construct homotopy colimits of pointwise cofibrant diagrams $D\col\mcI\ra\mcM_0$ in the following way. From $\bbI$ one may construct a cosimplicial resolution $\widetilde M=\bbI\otimes M$ of any cofibrant object $M\in\mcM$. This cosimplicial resolution is functorial and thus gives a cosimplicial resolution of $D$. With respect to such cosimplicial resolutions the homotopy colimit $\hocolim_\mcI D$ is thus given as
\[\hocolim\nolimits_\mcI D=\int\nolimits^{(n,i)\in\Delta\times\mcI}N(i\da\mcI)_n\cdot(\widetilde{D}i)^n=\int\nolimits^{(n,i)\in\Delta\times\mcI}N(i\da\mcI)_n\cdot\bbI^n\otimes Di.\]
Viewing $\int\nolimits^{n\in\Delta}N(i\da\mcI)_n\cdot\bbI^n$ as the image of $N(i\da\mcI)$ in $\mcV$ as in \eqref{eqn_embedding_of_sset} and denoting the ``simplicial homotopy enrichment'' \eqref{eqn_simplicial_enrichment} again as $\otimes$ one may rewrite the above as
\[\hocolim\nolimits_\mcI D=\int\nolimits^{i\in\mcI}N(i\da\mcI)\otimes Di,\]
which is the ``Bousfield-Kan formula'' for the homotopy colimit. We will elaborate on this in Proposition~\ref{prop_conical_to_weighted}.

\subsection{Homotopy weighted colimits}
In the following $\mcI$ will stand for a small ordinary (non-enriched) category and $\mcC$ for a small $\mcV$-category. We denote by $\mcI_\mcV$ the free $\mcV$-category generated by $\mcI$. It has the same set of objects and morphisms $\mcI_\mcV(i,j)=\mcI(i,j)\cdot\sfI$. Further we denote by $\mcI^i\in[\mcI^\op,\Set]$, $\mcC^c\in[\mcC^\op,\mcV]$ the contravariant hom-functors represented by $i$ and $c$, i.e.
\[\mcI^i=\mcI(-,i),\qquad \mcC^c=\mcC(-,c).\]
In particular $\Delta^n$ will stand for the standard simplicial $n$-simplex. Dually we denote by $\mcI_i\in[\mcI,\Set]$, $\mcC_c\in[\mcC,\mcV]$ the respective covariant hom-functors.

We say that an object $K\in\mcV$ is \emph{$\mcM$-flat} if the functor $K\otimes-\col\mcM\ra\mcM$ preserves both cofibrations and trivial cofibrations. More generally a map $K\ra L$ in $\mcV$ will be called \emph{$\mcM$-flat} if the map from the pushout to the bottom right corner in the following square
\[\xymatrix{
K\otimes X \ar[r] \ar[d] & K\otimes Y \ar[d] \\
L\otimes X \ar[r] & L\otimes Y
}\]
is a (trivial) cofibration whenever $X\ra Y$ is a (trivial) cofibration. By the compatibility axiom for the model $\mcV$-category $M$ all cofibrations in $\mcV$ are $\mcM$-flat. The canonical map $0\ra\sfI$ is also $\mcM$-flat so that the unit $\sfI$ is $\mcM$-flat. Also note that if the unit $\sfI$ is cofibrant then a map is $\mcV$-flat if and only if it is a cofibration. There might still be more $\mcM$-flat maps for a particular $\mcM$.

\begin{sdefinition}
We say that a small $\mcV$-category $\mcC$ is \emph{locally $\mcM$-flat} if all the hom-objects $\mcC(c,c')$ are $\mcM$-flat. If in addition each $\id_c\col\sfI\ra\mcC(c,c)$ is an $\mcM$-flat map we say that $\mcC$ is \emph{locally${}_*$ $\mcM$-flat}. The star stands for ``pointed'' since the endo-objects are required to be $\mcM$-flat as objects of $\sfI\da\mcV_0$.
\end{sdefinition}

Let $\mcC$ be now a small locally $\mcV$-flat $\mcV$-category. Then there exists a projective model structure on the $\mcV$-category $[\mcC^\op,\mcV]$ of all $\mcV$-functors $\mcC^\op\ra\mcV$ for which fibrations and weak equivalences are pointwise, see \cite{Stephan}. The set of generating (trivial) cofibrations is given by
\[\{\mcC^c\otimes K\rightarrowtail\mcC^c\otimes L\ |\ \textrm{$c\in\mcC$, $K\rightarrowtail L$ is a generating (trivial) cofibration}\}.\]
These conditions include the case $\mcC=\mcI_\mcV$ for which $[\mcI_\mcV^\op,\mcV]_0\cong[\mcI^\op,\mcV_0]$ since any coproduct of $\sfI$'s is $\mcV$-flat. If $\mcM$ is cofibrantly generated there is also a projective model structure on $[\mcC^\op,\mcM]$ whenever $\mcC$ is locall $\mcM$-flat.

\begin{sdefinition}
Among the colimits $W\star_\mcC D$ of diagrams $D\col\mcC\ra\mcM$ weighted by $W\col\mcC^\op\ra\mcV$ we distinguish those with the diagram $D$ pointwise cofibrant and the weight $W$ cofibrant in the projective model structure on $[\mcC^\op,\mcV]$. We call such colimits \emph{homotopy weighted colimits}. As a particular case the hom-functor $\mcC^c$ is cofibrant and $\mcC^c\star_\mcC D\cong Dc$ by the Yoneda lemma. We also say that $K\otimes M$ is a \emph{homotopy tensor} if both $K\in\mcV$ and $M\in\mcM$ are cofibrant (this is a special case of a homotopy weighted colimit for $\mcC=[0]_\mcV$ and $D0=M$, $W0=K$).
\end{sdefinition}

We will now study the relationship between homotopy weighted colimits and conical (non-enriched) homotopy colimits.

Let $\mcI$ first be an ordinary category. We will use $W\star_\mcI D$ to denote $\overline W\star_{\mcI_\mcV}\overline D$ where $\overline W$ and $\overline D$ are the unique $\mcV$-functors extending $W$ and $D$ so that it is in fact a weighted colimit and for $D$ pointwise cofibrant and $W$ cofibrant even a homotopy weighted colimit. The translation \eqref{eqn_weighted_colimit_to_coend} of every weighted colimit to a coend simplifies in this case to a non-enriched coend $\smallint^\mcI(W\otimes D)$ of a bifunctor
\[W\otimes D\col\mcI^\op\times\mcI\ra\mcM_0.\]
The following proposition says that conical homotopy colimits are special cases of homotopy weighted colimits.

\begin{proposition} \label{prop_conical_to_weighted}
If $W\in[\mcI^\op,\mcV_0]$ is an arbitrary cofibrant replacement of the constant diagram $\const_{\mcI^\op}\sfI$ with value $\sfI$ and $D\in[\mcI,\mcM_0]$ a pointwise cofibrant diagram then $W\star_\mcI D$ is weakly equivalent to $\hocolim_\mcI D$.
\end{proposition}

In terms of enriched categories the proposition says that the homotopy colimit of $D$ is weakly equivalent to $\overline W\star_{\mcI_\mcV}\overline D$. We may thus view conical homotopy colimits as special cases of homotopy weighted colimits namely those whose indexing $\mcV$-category $\mcC$ is a free $\mcV$-category $\mcI_\mcV$ on an ordinary category $\mcI$ and whose weight is a cofibrant replacement of $\overline{\const_{\mcI^\op}\sfI}$ (which only makes sense for such indexing $\mcV$-categories).

\begin{proof}
We will use the cosimplicial frame $\widetilde M=\bbI\otimes M\in\cosimp\mcM$ with respect to which the homotopy colimit of $D$ becomes
\[\hocolim\nolimits_\mcI D=N(-\da\mcI)\cdot_{\Delta\times\mcI}\widetilde D=N(-\da\mcI)\cdot_{\Delta\times\mcI}(\bbI\otimes D)\cong(N(-\da\mcI)\cdot_\Delta\bbI)\star_\mcI D.\]
Checking on the generating (trivial) cofibrations it is easy to see that
\[-\star_\mcI D\col[\mcI_\mcV^\op,\mcV]\lra\mcM\]
is left Quillen with respect to the projective model structure on $[\mcI_\mcV^\op,\mcV]$. Since $W$ is a cofibrant replacement of $\const_{\mcI^\op}\sfI$ we are left to show that $N(-\da\mcI)\cdot_\Delta\bbI$ is one too since then they must be weakly equivalent and taking the weighted colimit of $D$ yields the required weak equivalence.

To prove that $N(i\da\mcI)\cdot_\Delta\bbI\ra\bbI^0\simeq\sfI$ is a weak equivalence for each $i\in\mcI$ we use the fact that the embedding \eqref{eqn_embedding_of_sset} of $\sSet_0$ into $\mcV_0$ given by $-\cdot_\Delta\bbI$ is left Quillen and the natural map $N(i\da\mcI)\ra*$ is a weak equivalence (between cofibrant objects). As $N(-\da\mcI)$ is cofibrant in the projective model structure on $[\mcI^\op,\sSet_0]$ so is $N(-\da\mcI)\cdot_\Delta\bbI$ since the pointwise embedding
\[-\cdot_\Delta\bbI\col[\mcI^\op,\sSet_0]\lra[\mcI^\op,\mcV_0]\]
is again left Quillen (it maps $\mcI^i\cdot K$ to $\mcI^i\cdot(K\cdot_\Delta\bbI)$).
\end{proof}

In the opposite direction we will show how to decompose every homotopy weighted colimit into a conical homotopy colimit and homotopy tensors. For the following theorem we make the notation
\[\mcC(c_0,\ldots,c_n)=\mcC(c_{n-1},c_n)\otimes\cdots\otimes\mcC(c_0,c_1).\]

\begin{theorem} \label{thm_homotopy_weighted_colimits}
Suppose that the category $\mcC$ is locally${}_*$ $\mcM$-flat. Then homotopy weighted colimits in $\mcM$ may be obtained up to a weak equivalence from conical homotopy colimits and homotopy tensors. Explicitly
\begin{align*}
W\star_\mcC D & \simeq\Bigg|n\longmapsto\coprod_{c_0,\ldots,c_n\in\ob\mcC}Wc_n\otimes\mcC(c_0,\ldots,c_n)\otimes Dc_0\Bigg| \\
& \simeq\hocolim_{\substack{
[n]\in\Delta^\op \\
\makebox[0pt]{$\scriptstyle c_0,\ldots,c_n\in\ob\mcC$}
}}\ Wc_n\otimes\mcC(c_0,\ldots,c_n)\otimes Dc_0.
\end{align*}
\end{theorem}

In the second decomposition the indexing category is the category of ordered families of objects of $\mcC$ (or rather of $\delta\mcC$, the discrete category with the same set of objects as $\mcC$). Explicitly the morphisms from $\big([n],(c_0,\ldots,c_n)\big)$ to $\big([m],(c_0',\ldots,c_m')\big)$ are those ordinal maps $f\col[m]\ra[n]$ for which $c_{f(i)}=c_i'$ for all $i=0,\ldots,m$.

The homotopy colimit version (which may be also viewed as the ``geometric realization'' of the $\Delta_+^\op$-subdiagram of the $\Delta^\op$-diagram from the first decomposition) holds true whenever $\mcC$ is only locally $\mcM$-flat. The pointed version of the local $\mcM$-flatness enables one to replace the homotopy colimit by the geometric realization.

\begin{remark}
Theorem~\ref{thm_homotopy_weighted_colimits} in the case that both $W$ and $D$ are merely pointwise cofibrant says the following. While the weighted colimit $W\star_\mcC D$ is not homotopically correct the right hand side is and thus provides a model for the derived weighted colimit which we denote as $-\mathbin{\widehat\star_\mcC}-=-\star_\mcC^\cof-$. Therefore one may write for pointwise cofibrant $W$ and $D$:
\[W\mathbin{\widehat\star_\mcC}D\simeq\hocolim_{\substack{
[n]\in\Delta^\op \\
c_0,\ldots,c_n\in\ob\mcC
}}Wc_n\otimes\mcC(c_0,\ldots,c_n)\otimes Dc_0.\]
As an example we get a formula for the conical homotopy colimit in simplicial model categories, i.e.~for $\mcV=\sSet$. For an ordinary category $\mcI$ and $W=*$ the theorem reads that $\hocolim\nolimits_\mcI D=(\const_{\mcI^\op}*)\mathbin{\widehat\star_\mcI}D$ is weakly equivalent to
\[\Bigg|n\longmapsto\coprod_{i_0,\ldots,i_n}\mcI(i_0,\ldots,i_n)\times Di_0\Bigg|\cong\Bigg|n\longmapsto\coprod_{i_0\ra\cdots\ra i_n}Di_0\Bigg|.\]
This is the Bousfield-Kan formula for homotopy colimits (up to the usual problems with an ambiguous definition of a nerve).
\end{remark}

\begin{remark}
One may organize $\mcC(c_0,\ldots,c_n)$ into a simplicial object $\Delta\mcC\in\simp[\mcC^\op\otimes\mcC,\mcV]$ mapping
\[(c,c')\mapsto\coprod_{c_0,\ldots,c_n}\mcC(c,c_0,\ldots,c_n,c')\]
and the construction from the theorem then becomes $|\Delta\mcC|\star_{\mcC^\op\otimes\mcC}(W\otimes D)$. The weight $|\Delta\mcC|$ is a cofibrant replacement of the hom-bifunctor $\hom_\mcC\col\mcC^\op\otimes\mcC\ra\mcV$ which gives the ordinary coend $\smallint^\mcC F\cong\hom_\mcC\star_{\mcC^\op\otimes\mcC}F$. It is therefore true that $|\Delta\mcC|$ may be used to define a derived coend for general bifunctors (other than tensor products of functors). We will make a further comment on this, the important thing to note is that $\Delta\mcC=\sfB(y_{\mcC^\op\otimes\mcC})$.
\end{remark}

For the rest of this section we will concentrate on proving Theorem~\ref{thm_homotopy_weighted_colimits} and will give a few applications in the next section. We start with the non-homotopy version of our theorem, i.e.~the isomorphism $W\star_\mcC D\cong\smallint^\mcC(W\otimes D)$ and its refinement \eqref{eqn_weighted_colimit_to_coequalizer},
\[\xymatrix{
W\star_\mcC D\cong\mathop{\mathrm{coeq}}\Big(\displaystyle\coprod_{c_0,c_1}Wc_1\otimes\mcC(c_0,c_1)\otimes Dc_0 \ar@<-2pt>[r] \ar@<2pt>[r] & \displaystyle\coprod\limits_c Wc\otimes Dc\Big).
}\]
Note that this coequalizer is reflexive. In homotopy situations one should replace reflexive coequalizers by geometric realizations of simplicial objects.

\begin{observation}
Let $F\col\mcC^\op\otimes\mcC\ra\mcM$ be a $\mcV$-bifunctor with $\mcC$ a small $\mcV$-category and $\mcM$ a tensored $\mcV$-category. The coend $\smallint\nolimits^\mcC F$ of $F$ is isomorphic to the colimit of the diagram $\sfB F\col\Delta^\op\ra\mcM_0$ given by
\[n\longmapsto\coprod_{c_0,\ldots,c_n\in\mcC}\mcC(c_0,\ldots,c_n)\otimes F(c_n,c_0).\]
Accordingly there is an augmentation $\varepsilon\col\sfB F\ra\smallint^\mcC F$.
\end{observation}

\begin{proof}
This follows from the inclusion of the parallel pair $d_0,d_1\col[1]\rightrightarrows[0]$ into $\Delta^\op$ being right cofinal, i.e.~the colimit of $\sfB F$ is the coequalizer of the two arrows induced by $d_0,d_1$.
\end{proof}

If $\mcM$ is now a model $\mcV$-category we may consider various homotopy versions of the colimit of $\sfB F$. In the classical situation $\mcV=\sSet$ the geometric realization will be called the homotopy coend of $F$ (it is called a simplicially coherent coend in \cite{CordierPorter} and in the non-enriched context of \cite{Weibel} a homotopy coend). As explained in \eqref{eqn_geometric_realization} we may define geometric realization over arbitrary $\mcV$ using the cosimplicial resolution $\bbI$ of $\sfI$, i.e.
\[|\sfB F|=\bbI\star_\Delta\sfB F,\]
the weighted colimit with the weight $\bbI$. There is a transformation $\bbI\ra\const_\Delta\sfI$ of weights from $\bbI$ to the constant cosimplicial object on $\sfI$. On weighted colimits this induces a canonical map
\[|\sfB F|\ra\colim\nolimits_{\Delta^\op}\sfB F\cong\smallint\nolimits^\mcC F.\]

In the same manner we may define the homotopy colimit of $\sfB F$ using a cofibrant replacement of $\const_\Delta\sfI$ in the projective model structure. We will see however that in our situation $\sfB F$ is Reedy cofibrant and thus the homotopy colimit and the geometric realization are weakly equivalent (see Theorem~19.8.7 of \cite{Hirschhorn}).

\begin{proposition} \label{prop_Reedy_cofibrancy}
Let $\mcC$ be a small locally${}_*$ $\mcM$-flat $\mcV$-category. If $F$ is pointwise cofibrant then $\sfB F$ is Reedy cofibrant in $\simp\mcM_0$ and hence $|\sfB F|$ is cofibrant in $\mcM$.
\end{proposition}

\begin{proof}
Postponed to a later section.
\end{proof}

In fact it holds more generally that if $F\Ra F'$ is a pointwise (trivial) cofibration then $\sfB F\Ra\sfB F'$ is a Reedy (trivial) cofibration.

Observe that the association $F\mapsto|\sfB F|$ is a functor $[\mcC^\op\otimes\mcC,\mcM]\ra\mcM$. According to the previous proposition it maps pointwise cofibrant diagrams to cofibrant objects. Moreover it commutes with all cocontinuous $\mcV$-functors. This is clear by a direct inspection since it is defined by weighted colimits.

Now we prove the most important step which for $W=\mcC^{c'}$ gives $|\Delta\mcC|\xlra\sim\hom_\mcC$ as promised in one of the remarks following our main theorem.

\begin{lemma}
Let $\mcC$ be a small locally${}_*$ $\mcM$-flat $\mcV$-category. For any pointwise cofibrant weight $W\col\mcC^\op\ra\mcV$ the natural map
\[|\sfB(W\otimes\mcC_c)|\lra\smallint\nolimits^\mcC(W\otimes\mcC_c)\cong Wc\]
is a weak equivalence.
\end{lemma}

\begin{proof}
We recall that $\sfB(W\otimes\mcC_c)$ admits an augmentation to $\smallint^\mcC(W\otimes\mcC_c)\cong Wc$ whose geometric realization is the first map in the composition
\[|\sfB(W\otimes\mcC_c)|\lra|\const_{\Delta^\op}Wc|\cong\bbI^0\otimes Wc\xlra{\sim}Wc.\]
Since the second map is a weak equivalence by the monoidal model category axioms it remains to study the geometric realization of the augmentation. This is taken care of by Proposition~\ref{prop_extra_degenerate} since the augmented simplicial object $\sfB(W\otimes\mcC_c)$ mapping
\[n\longmapsto\coprod_{c_0,\ldots,c_n\in\ob\mcC}Wc_n\otimes\mcC(c,c_0,\ldots,c_n)\]
is Reedy cofibrant and admits an extra degeneracy $s_{-1}$ given by ``repeating'' $c$. On the obvious summands it is given by
\[Wc_n\otimes\mcC(c,c_0,\ldots,c_n)\ra Wc_n\otimes\mcC(c,c,c_0,\ldots,c_n)\]
(a tensor of the domain with the ``identity'' map $\id_c\col\sfI\ra\mcC(c,c)$).
\end{proof}

By the previous lemma there is a pointwise weak equivalence
\[|\sfB(W\otimes\mcC_\bullet)|\xlra{\sim}W\]
where $\mcC_\bullet$ denotes the Yoneda embedding $\mcC\ra[\mcC^\op,\mcV]$. Since the values of $W\otimes\mcC_\bullet$ are $Wc\otimes\mcC_\bullet c'=Wc\otimes\mcC^{c'}$ they are cofibrant and this map is a cofibrant replacement for a pointwise cofibrant $W$.

\begin{proof}[Proof of Theorem~\ref{thm_homotopy_weighted_colimits}]
Applying $-\star_\mcC D$ to the above cofibrant replacement yields the following weak equivalence since $W$ is assumed to be cofibrant already.
\[|\sfB(W\otimes D)|\cong|\sfB((W\otimes\mcC_\bullet)\star_\mcC D)|\cong|\sfB(W\otimes\mcC_\bullet)|\star_\mcC D\xlra{\sim}W\star_\mcC D\]
The first isomorphism follows from the ``Yoneda lemma'' $\mcC_\bullet\star_\mcC D\cong D$ (see \cite{Kelly}, p.~53) and the second from $-\star_\mcC D$ being cocontinuous.
\end{proof}

The cocontinuity of $-\star_\mcC D$ also implies the following principle. Any homotopy weighted colimit decomposition of $W$ yields a decomposition of $W\star_\mcC D$. In the proceeding we will therefore concentrate on decompositions of weights.

\section{Examples for classical enrichments}

\subsection{Homotopy density}

We would like to compare the expressive power of homotopy weighted colimits with that of conical homotopy colimits in specific examples $\mcV=\sSet,\sAb,\Ch_+,\Ch,\Cat$. We will make this comparison through the notion of homotopy density. A pointwise cofibrant $\mcV$-functor $J\col\mcC\ra\mcM$ is called \emph{homotopy dense} if for every fibrant $M\in\mcM$ the canonical map
\[\mcM(J-,M)^\cof\star_\mcC J\lra M\]
is a weak equivalence where $(-)^\cof$ denotes the projective cofibrant replacement. We call the homotopy weighted colimit on the left the \emph{canonical homotopy weighted colimit associated with $J$ and $M$}. This map is easily seen to be the derived counit $L^\cof RM\ra M$ of the Quillen adjunction
\[\xymatrix{
L\col[\mcC^\op,\mcV] \ar@<2pt>[r] & \mcM\loc R \ar@<2pt>[l]
}\]
with $L$ the unique cocontinuous extension of $J$, given by $LW=W\star_\mcC J$. The right adjoint is $RM=\mcM(J-,M)$. We say that a Quillen adjunction is a \emph{homotopy reflection} if the derived counit is a weak equivalence on fibrant objects. Therefore $J$ is homotopy dense if and only if $L\dashv R$ is a homotopy reflection. A prime example is that of a Yoneda embedding. The canonical weight associated with $y_\mcC\col\mcC\ra[\mcC^\op,\mcV]$ is $[\mcC^\op,\mcV](y_\mcC-,X)\cong X$ and consequently $y_\mcC$ is homotopy dense since
\[X^\cof\star_\mcC y_\mcC\cong X^\cof\xlra{\sim}X.\]

A slightly different situation occurs in the non-enriched context. We will not speak however about the presheaf category $[\mcI^\op,\Set]$ but rahter about the simplicial presheaf category $[\mcI_\mcV^\op,\sSet]$ or more generally $[\mcC^\op,\sSet]$ for a small simplicial category $\mcC$. The reason is the ``simplicial homotopy enrichment'' of any model category given by \eqref{eqn_simplicial_enrichment}. In this case it is not true in general that every fibrant presheaf $X$ is weakly equivalent to the ``canonical conical homotopy colimit'' with respect to $\mcC$ or more precisely $\mcC_0$. Let us define this concept in detail. We assume that $\mcM$ is a simplicial model category for simplicity but it is not really necessary. We say that a pointwise cofibrant functor $J\col\mcI\ra\mcM_0$ is \emph{conical homotopy dense} if for each fibrant $M\in\mcM$ the canonical map
\[\mcM_0(J-,M)^\cof\star_\mcI J\lra M\]
is a weak equivalence. We will show in general that while $\mcC_0\subseteq[\mcC^\op,\sSet]_0$ is not conical homotopy dense its simple modification is. Thinking of $\mcC_0$ as the category of representables the modification $(\Delta\ltimes\mcC)_0$ consists of the tensor products $\Delta^k\otimes\mcC^c$ for varying $k$ and $c$.

The non-homotopy density is typically expressed through conical colimits indexed by the category of elements. We will now do the same translation in the homotopy context. Namely we will show that the conical homotopy colimit of the composition
\[\mcJ=J\da M\xlra{P_M}\mcI\xlra{J}\mcM_0\]
is weakly equivalent to $\mcM_0(J-,M)^\cof\star_\mcI J$. Computing with the simplicial enrichment $\hocolim\nolimits_\mcJ JP_M$ is given by
\begin{equation} \label{eqn_hocolim_of_representables}
N(-\da\mcJ)\star_\mcJ JP_M\cong(P_M)_!N(-\da\mcJ)\star_\mcI J\cong N(-\da P_M)\star_\mcI J.
\end{equation}
Since $P_M$ is a discrete Grothendieck fibration we have
\[N(-\da P_M)\simeq N(P_M^{-1}(-))=\mcM_0(J-,M).\]
Finally $N(-\da P_M)$ is cofibrant and this finishes the proof of the weak equivalence
\[\mcM_0(J-,M)^\cof\star_\mcI J\simeq\hocolim\nolimits_\mcJ JP_M.\]

We give now an explicit formula using the following description of the weight $N(-\da P_M)$,
\[N(-\da P_M)\cong\int^iN(i\da P_m)\otimes\mcI^i\cong\int^i\Bigg|n\mapsto\coprod_{\substack{
i\ra i_0\ra\cdots\ra i_n \\
Ji_n\ra M}}
\mcI^i\Bigg|\cong\Bigg|n\mapsto\coprod_{\substack{
i_0\ra\cdots\ra i_n \\
Ji_n\ra M}}
\mcI^{i_0}\Bigg|.\]
Applying the weighted colimit functor $-\star_\mcI J$ we obtain
\[\hocolim\nolimits_{J\da M}JP_M\cong\Bigg|n\longmapsto\coprod_{\substack{
i\ra i_0\ra\cdots\ra i_n \\
Ji_n\ra M}}
Ji_0\Bigg|=\Bigg|n\longmapsto\coprod_{i_0\ra\cdots\ra i_n\ra M}i_0\Bigg|\]
where in the last expression we omitted $J$ under the assumption that it is a full embedding.

\subsection{Simplicial presheaves on an ordinary category}
For $\mcV=\sSet$ and an ordinary category $\mcI$, Theorem~\ref{thm_homotopy_weighted_colimits} with $W\in[\mcI^\op,\sSet]$ (pointwise) cofibrant says
\[W\cong W\star_\mcI\mcI_\bullet\simeq\Biggl|n\longmapsto\coprod_{i_0,\ldots,i_n}Wi_n\times\mcI(i_0,\ldots,i_n)\times\mcI^{i_0}\Biggr|.\]
We may rewrite $Wi_n\times\mcI(i_0,\ldots,i_n)$ as the simplicial set whose set of $m$-simplices is the set of all chains $\mcI^{i_0}\ra\cdots\mcI^{i_n}\ra W_m$ in $[\mcI^\op,\Set]$. Since the double geometric realization is isomorphic to the geometric realization of the diagonal we obtain a formula
\begin{equation} \label{eqn_first_replacement}
W\simeq\Bigg|n\longmapsto\Bigg|m\longmapsto\coprod_{\mcI^{i_0}\ra\cdots\mcI^{i_n}\ra W_m}\mcI^{i_0}\Bigg|\Bigg|\cong\Bigg|n\longmapsto\coprod_{\mcI^{i_0}\ra\cdots\mcI^{i_n}\ra W_n}\mcI^{i_0}\Bigg|.
\end{equation}
which presents $W$ as a homotopy colimit of representables. In particular if each $Wi$ is discrete the canonical map
\[\hocolim\nolimits_{J\da W}JP_W\ra W\]
is a weak equivalence but not in general. In other words the Yoneda embedding $y_\mcI\col\mcI\ra[\mcI^\op,\sSet]$ is not conical homotopy dense even though every simplicial presheaf is a conical homotopy colimit of representables (presheaves in the image of $y_\mcI$). This formula was obtained by Dugger in \cite{Dugger}. We will also derive his second cofibrant replacement formula in a greater generality -- for simplicial indexing categories.

\subsection{Simplicial presheaves on a simplicial category}
Still restricting to the case $\mcV=\sSet$ one has for $K\in\sSet$ the following weak equivalences which are natural in $K$:
\[K\cong\colim_{\Delta^k\ra K}\Delta^k\xlla{\sim}\hocolim_{\Delta^k\ra K}\Delta^k\xlra{\sim}\hocolim\nolimits_{\Delta K}*\]
where the indexing category $\Delta K$ is the category of simplices of $K$ (the category of elements), see \cite{Hirschhorn}, Corollary 19.9.2 with $p=\id_K$. Therefore if $M\in\mcM$ is cofibrant in a simplicial model category $\mcM$ we get weak equivalences
\[K\otimes M\xlla{\sim}\hocolim_{\Delta^k\ra K}\Delta^k\otimes M\xlra{\sim}\hocolim\nolimits_{\Delta K}M.\]
Consequently (homotopy) tensors are ``generated'' by homotopy colimits. More precisely we may replace each tensor up to weak equivalence by a conical homotopy colimit (of a constant diagram with value $M$). We note however that this zig-zag of weak equivalences is natural in $K$ and $M$ but \emph{not} in $K\otimes M$. As a result it is not true that every simplicial presheaf is weakly equivalent to a conical homotopy colimit of representables as the following example shows. One may use the first decomposition of $K\otimes M$ (which is still reasonably natural) to show that every simplicial presheaf is a conical homotopy colimit of presheaves of the form $\Delta^k\otimes\mcC^c$. We will prove this result by a different method which will yield a stronger result namely the conical homotopy density of the presheaves of the said form.

\begin{sexample}[Representables are not conical homotopy dense]
Consider the category $\mcC$ with two objects $0$ and $1$ whose only non-identity morphisms form $\mcC(0,1)=K$ with a single vertex and not weakly contractible. Then the representables form a category $\mcC_0=[1]$ and we will now show that the terminal funcor $\mcC^\op\ra\sSet$ may not be weakly equivalent to a homotopy colimit of representables. Suppose that $W$ is the homotopy colimit of $\mcI\xra{P}\mcC_0\ra[\mcC^\op,\sSet]_0$. Then according to \eqref{eqn_hocolim_of_representables},
\begin{align*}
W1 & \cong N(1\da P), \\
W0 & \cong\Big(K\times N(1\da P)\Big)\sqcup_{N(1\da P)}N(0\da P).
\end{align*}
Thus if $W1\simeq *$ then $W0\simeq K\vee N(0\da P)$ and may not be weakly contractible.
\end{sexample}

We remark here that as a result of \eqref{eqn_second_replacement} every weight $W\in[\mcC^\op,\sSet]$ over a general simplicial category $\mcC$ is weakly equivalent to a conical homotopy colimit of a homotopy coherent diagram of representables and therefore also to a conical homotopy colimit of a simplicial diagram of representables.

We would like to prove now the above claim about conical homotopy density of the full subcategory $\Delta\ltimes\mcC\subseteq[\mcC^\op,\sSet]$ consisting of the simplicial presheaves of the form $\Delta^k\otimes\mcC^c$ which we will for simplicity denote by $(k,c)$. In the proceeding we will need a concrete description of the morphism spaces
\[\Delta\ltimes\mcC\Big((k',c'),(k,c)\Big)=\sSet(\Delta^{k'},\Delta^k\times\mcC(c',c)).\]
As usual we denote by $(\Delta\ltimes\mcC)_0$ the underlying ordinary category of $\Delta\ltimes\mcC$.

\begin{theorem}
The full subcategory $(\Delta\ltimes\mcC)_0\subseteq[\mcC^\op,\sSet]_0$ is conical homotopy dense.
\end{theorem}

\begin{proof}
We have a diagram
\[\xymatrix@=10pt{
\mcC \ar[dr]_-G \ar@/^10pt/[drrr]^{y_\mcC} \\
& \Delta\ltimes\mcC \ar[rr]^-J & & [\mcC^\op,\sSet] \\
(\Delta\ltimes\mcC)_0 \ar[ur]^-H \ar@/_10pt/[urrr]_{J_0}
}\]
Let $W\in[\mcC^\op,\sSet]$ be a cofibrant weight. The canonical weights are then respectively $W$ on $\mcC$, $V$ on $\Delta\ltimes\mcC$ and the canonical \emph{conical} weight $V_0$ on $(\Delta\ltimes\mcC)_0$ where
\[V\col(k,c)\longmapsto(Wc)^{\Delta^k},\qquad V_0\col(k,c)\longmapsto W_kc.\]
They are related by morphisms $W\xra\cong G^*V$ and $(V_0)^\cof\ra H^*V$ (the inclusion of the 0-skeleton). By adjunction we obtain morphisms
\[G_!W\lra V\lla H_!(V_0)^\cof\]
of weights on $\mcC$. Our aim will be to show that they are weak equivalences as
\[W\cong W\star_\mcC y_\mcC\cong G_!W\star_{\Delta\ltimes\mcC}J\]
will then be weakly equivalent to
\[\hocolim\nolimits_{J_0\da W}J_0P_W\cong(V_0)^\cof\star_{(\Delta\ltimes\mcC)_0}J_0\simeq H_!(V_0)^\cof\star_{\Delta\ltimes\mcC}J.\]

Since in $\Delta\ltimes\mcC$ the object $(k,c)$ is simplicially homotopy equivalent to $(0,c)$ and since any simplicial functor preserves simplicial homotopies it is only necessary to test weak equivalences on objects $(0,c)$, i.e.~on the image of $G$. But $G_!W\ra V$ is even an isomorphism on the image of $G$ since $G$ is full. For the second weight we compute using the standard formula for the left Kan extension (see (4.18) in \cite{Kelly})
\[H_!(V_0)^\cof(0,c)=\int\nolimits^{(k_0,c_0)}\Delta\ltimes\mcC\big((0,c),(k_0,c_0)\big)\otimes(V_0)^\cof(k_0,c_0).\]
Since $\Delta^n\otimes(0,c)=(n,c)$ in $\Delta\ltimes\mcC$ we may write
\[\Delta\ltimes\mcC\big((0,c),(k_0,c_0)\big)=\Big|n\mapsto(\Delta\ltimes\mcC)_0\big((n,c),(k_0,c_0)\big)\Big|\]
and thus obtain
\begin{align*}
H_!(V_0)^\cof(0,c) & \cong\Bigg|n\longmapsto\int\nolimits^{(k_0,c_0)}(\Delta\ltimes\mcC)_0\big((n,c),(k_0,c_0)\big)\otimes (V_0)^\cof(k_0,c_0)\Bigg| \\
& \cong\Big|n\longmapsto(V_0)^\cof(n,c)\Big|\simeq\Big|n\longmapsto W_nc\Big|\cong Wc=V(0,c)
\end{align*}
where the point of the weak equivalence in the last line is that the geometric realization is homotopy invariant and we may thus leave out the cofibrant replacement. 
\end{proof}

Since in this case $(\Delta\ltimes\mcC)_0$ is a full subcategory of $[\mcC^\op,\sSet]_0$ we may write the homotopy colimit in the following form
\[W\simeq\hocolim\nolimits_{J_0\da W} J_0P_W\cong\Bigg|n\longmapsto\coprod_{\makebox[70pt]{$\scriptstyle\Delta^{k_0}\otimes\mcC^{c_0}\ra\cdots\ra\Delta^{k_n}\otimes\mcC^{c_n}\ra W$}}\Delta^{k_0}\otimes\mcC^{c_0}\Bigg|.\]
This is a generalization of a cofibrant replacement of \cite{Dugger} from ordinary categories to simplicial ones. In this case we will write down explicitly the corresponding decomposition result for homotopy weighted colimits,
\begin{equation} \label{eqn_second_replacement}
W\star_\mcC D\simeq\Bigg|n\longmapsto\coprod_{\makebox[70pt]{$\scriptstyle\Delta^{k_0}\otimes\mcC^{c_0}\ra\cdots\ra\Delta^{k_n}\otimes\mcC^{c_n}\ra W$}}\Delta^{k_0}\otimes Dc_0\Bigg|.
\end{equation}
A nice formulation of this decomposition is the following.

\begin{scorollary}
Let $\mcM$ be a simplicial model category. Then for a full subcategory of $\mcM^\cof$ closed under weak equivalences it is equivalent to be closed under conical homotopy colimits and to be closed under homotopy weighted colimits.\qed
\end{scorollary}

We may in fact define $\sSet$-enriched homotopy weighted colimits on an arbitrary model category $\mcM_0$ (without a simplicial enrichment) and the claim still holds by the very same homotopy colimit decomposition of weights.

\subsection{A discrete presentation for simplicial presheaves on a simplicial category}
We start with a simple observation.

\begin{lemma} \label{lem_conical_dense_implies_simplicial}
Let $J\col\mcC\ra\mcM$ be an embedding of a full subcategory of a simplicial model category~$\mcM$. If $J$ is conical homotopy dense then it is also homotopy dense.
\end{lemma}

\begin{proof}
We need to show that for each fibrant $M$ the map
\[\mcM(J-,M)\mathbin{\widehat\star_\mcI}J\lra M\]
is a weak equivalence. Decomposing $\mcM(J-,M)$ into simplices we obtain
\[\mcM(J-,M)\mathbin{\widehat\star_\mcI}J\cong\Big|n\mapsto\mcM_0(J-,M^{\Delta^n})\mathbin{\widehat\star_\mcI}J\Big|\xlra{\sim}\Big|n\mapsto M^{\Delta^n}\Big|\xlra\sim M\]
with the first weak equivalence coming from the conical homotopy density. The second is implied by the simplicial object being homotopy constant.
\end{proof}

With the embedding $J_0\col(\Delta\ltimes\mcC)_0\ra[\mcC^\op,\sSet]_0$ as above we consider the corresponding Quillen adjunction
\[\xymatrix{
L\col[(\Delta\ltimes\mcC)_0^\op,\sSet] \ar@<2pt>[r] & [\mcC^\op,\sSet]\loc R. \ar@<2pt>[l]
}\]
Since $\Delta\ltimes\mcC$ is conical homotopy dense it is also homotopy dense and this Quillen adjunction is a homotopy reflection. Since $R$ preserves all conical homotopy colimits and preserves and reflects all weak equivalences one can localize $[(\Delta\ltimes\mcC)_0^\op,\sSet]$ with respect to the unit of the adjunction at the representables,
\[(\Delta\ltimes\mcC)_0^{(k,c)}\lra RL(\Delta\ltimes\mcC)_0^{(k,c)},\]
to obtain a Quillen equivalence. This yields a presentation result for $[\mcC^\op,\sSet]$, namely that it is Quillen equivalent to a localization of the category of simplicial presheaves on an ordinary category $(\Delta\ltimes\mcC)_0$. Consequently to define a simplicial presheaf on $\mcC$ is ``up to homotopy'' the same as to define a local $(\Delta\ltimes\mcC)_0$-presheaf.

\begin{remark}
It is not too hard to show that a pointwise fibrant presheaf $X$ is local if and only if all the canonical ``degeneracy'' maps $X(0,c)\ra X(k,c)$ are weak equivalences. If such a ``homotopy presheaf'' is given it is possible to rigidify it to an actual presheaf $L^\cof X$.
\end{remark}

\subsection{Presheaves over $\sAb$}

Let $\mcF\subseteq\sAb$ denote the full subcategory formed by the free abelian groups $\bbZ^p$. Here we think of $\sAb$ as a simplicial category. Note that $\hom_\mcF$ in fact takes values in discrete simplicial sets so that we may and will think of it as an ordinary category. We will now use the rigidification result of \cite{Badzioch} or \cite{Bergner} which says that the induced Quillen adjunction
\[\xymatrix{
L\col[\mcF^\op,\sSet]_\mathrm{loc} \ar@<2pt>[r] & \sAb\loc R \ar@<2pt>[l]
}\]
is a Quillen equivalence where the subscript ``$\mathrm{loc}$'' denotes a certain localization of the presheaf category. For a fibrant $X\in\sAb$ the derived counit is thus a weak equivalence and the same is therefore true for the non-localized model structure showing that $\mcF$ is homotopy dense and consequently $\Delta\times\mcF$ is conical homotopy dense (it is a product since $\mcF$ is an ordinary category). Again there results a weak equivalence
\[\Bigg|n\longmapsto\bigoplus_{\makebox[70pt]{$\scriptstyle\Delta^{k_0}\otimes\bbZ^{p_0}\ra\cdots\ra\Delta^{k_n}\otimes\bbZ^{p_n}\ra X$}}\Delta^{k_0}\otimes\bbZ^{p_0}\Bigg|\xlra{\sim}X\]
for every (fibrant) simplicial abelian group $X\in\sAb$ where the chains indexing the direct sum consist of ``tensor products'' of maps between $\Delta^{k_i}$ and homomorphisms between $\bbZ^{p_i}$ since they are given by the elements of $\Delta\times\mcF$. We conclude again that in $\sAb$ the homotopy tensors are generated from the unit $\bbZ$ by conical homotopy colimits.

We will now extend the presentation result to the presheaf category $[\mcC^\op,\sAb]$. Let $\mcF$ denote the free completion of $\mcC$ under finite direct sums, i.e.~the full subcategory of $[\mcC^\op,\sAb]$ formed by the finite sums of representables.

\begin{theorem}
The subcategory $\mcF$ is simplicially homotopy dense. Consequently $\Delta\ltimes\mcF$ is conical homotopy dense.
\end{theorem}

\begin{proof}
We give an outline of a proof based on the proofs of \cite{Badzioch} and \cite{Bergner} of the rigidification of algebras. In our case the theory will be multi-sorted and simplicially enriched and this is not covered by any of these papers.

We have a simplicial Quillen adjunction which we want to show to be a homotopy reflection.
\[\xymatrix{
L\col[\mcF^\op,\sSet] \ar@<2pt>[r] & [\mcC^\op,\sAb]\loc R. \ar@<2pt>[l]
}\]
It is rather straightforward to show that the counit $LRX\ra X$ is an isomorphism (especially in its equivalent form of $R$ being full and faithful). In effect $\mcF$ is simplicially dense. In the following we will need two simple observations about $R$ -- it preserves sifted colimits as it is given by the hom-functors out of finite direct sums of representables. More precisely we will need that $R$ preserves filtered colimits and the geometric realization. The sifted colimits and their homotopy versions are treated in \cite{RosickySifted}. The second observation about $R$ is that it preserves and reflects all weak equivalences. In particular the unit and the derived unit are weakly equivalent.

We localize $[\mcF^\op,\sSet]$ with respect to the set of all unit maps $\eta_W\col W\ra RLW$ with $W$ running over finite direct sums of representables, $W=\mcF^{F_1}\oplus\cdots\oplus\mcF^{F_q}$. Since all such $W$ and $RLW$ are cofibrant ($RLW\cong\mcF^{F_1\oplus\cdots\oplus F_q}$) and $L\eta_W$ are isomorphisms ($L$ is a reflection), $L^\cof\eta_W$ must be weak equivalences. Therefore we obtain a new Quillen adjunction
\[\xymatrix{
L\col[\mcF^\op,\sSet]_\mathrm{loc} \ar@<2pt>[r] & [\mcC^\op,\sAb]\loc R \ar@<2pt>[l]
}\]
whose unit is now a weak equivalence on finite direct sums of representables. Because $R$ preserves sifted colimits the unit is a weak equivalence on all cofibrant objects as by \eqref{eqn_first_replacement} conical homotopy colimits are generated by coproducts (which are in turn generated by finite coproducts and filtered colimits) and the geometric realization\footnote{Filtered colimits are homotopy invariant in $\sSet$ and $\sAb$ by a very simple argument (the generating cofibrations have both domains and codomains finitely generated). The same goes with the geometric realization in $\sSet$ since every simplicial object in this category is Reedy cofibrant (in $\sAb$ this is not true but we do not need this). Consequently every simplicial object in $[\mcF^\op,\sSet]_\mathrm{inj}$ with its injective model structure is also Reedy cofibrant and on such the geometric realization is homotopy invariant. Since this statement depends only on weak equivalences the same is true in the projective model structure.} -- the unit commutes with both.

We will now show that the derived counit at a fibrant $X$ is also a weak equivalence. Since $R$ reflects weak equivalences we apply it to the derived counit $\varepsilon^\cof$ to obtain the diagram
\[\xymatrix{
(RX)^\cof \ar[r]^-\eta_-\sim & RL(RX)^\cof \ar[r] \ar@{=}[d] & (RX)^\cof \ar[d]^\sim \\
& RL^\cof RX \ar[r]^-{R\varepsilon^\cof} & RX
}\]
where the composition across the top row is the identity and the unit is a (local) weak equivalence by the preceding. Consequently the map at the bottom is a weak equivalence and thus so is the derived counit. Since the derived counit is unchanged by the localization the same holds in the original Quillen adjunction and we may conclude that $\mcF$ is simplicially homotopy dense.
\end{proof}

Therefore homotopy weighted colimits over $\sAb$ are generated by conical homotopy colimits and tensors with objects of $\Delta\ltimes\mcF$. These are given by
\[\Delta^k\otimes(Dc_1\oplus\cdots\oplus Dc_p)\simeq Dc_1\oplus\cdots\oplus Dc_p\]
and are again (up to weak equivalence) generated by conical homotopy colimits.

\begin{scorollary}
Let $\mcM$ be a model $\sAb$-category. Then for a full subcategory of $\mcM^\cof$ closed under weak equivalences it is equivalent to be closed under conical homotopy colimits and to be closed under homotopy weighted colimits.\qed
\end{scorollary}

The Dold-Kan correspondence provides an equivalence of $\sAb$ with the category $\Ch_+$ of non-negatively graded abelian groups. This equivalence respects the model structures $\Ch_+$ and the monoidal structures up to weak equivalence. The same results we obtained for $\sAb$ are therefore true for $\Ch_+$. Moreover all of the above can be generalized to simplicial $R$-modules and non-negatively graded chain complexes of $R$-modules.

\subsection{The enrichment in $\Cat$ and $\Ch$.}

We will briefly describe the situation for the enrichment in the categorical model structure on $\Cat$ (also called the folk model structure) and for the category $\Ch$ of unbounded chain complexes. We begin with a general consideration. We recall the embedding \eqref{eqn_embedding_of_sset} of simplicial sets into any monoidal model category. The image of $K\in\sSet$ is then seen to be
\[K\otimes_\Delta\bbI\simeq(\hocolim\nolimits_{\Delta K}*)\otimes_\Delta\bbI=\hocolim\nolimits_{\Delta K}(*\otimes_\Delta\bbI)=\hocolim\nolimits_{\Delta K}\sfI.\]
and is therefore the homotopy colimit of the constant diagram at $\sfI$. In the case $\Cat$ the image of $K$ is precisely the fundamental groupoid of $K$. It is in fact easily seen that groupoids are exactly the conical homotopy colimit closure of the unit $[0]\in\Cat$. In order to obtain all categories one needs to start with all $[n]$ (or at least those with $n\leq2$). This yields a presentation
\[\xymatrix{
L\col[\Delta^\op,\sSet]_\mathrm{loc} \ar@<2pt>[r] & \Cat\loc R. \ar@<2pt>[l]
}\]

A similar situation occurs with unbounded chain complexes. The image of $K\in\sSet$ in $\Ch$ might be taken to be the chain complex $\bbZ K$ assosiated with the free abelian group on $K$. Therefore it only has homology in non-negative degrees and again one may show that complexes whose homology is concentrated in non-negative dimensions provide exactly the conical homotopy colimit completion of the unit $\bbZ$ (the truncation restricts to a natural weak equivalence on this subcategory and a conical homotopy colimit of chain complexes concentrated in non-negative dimensions is a chain complex of the same type). In particular $\bbZ[-1]$ does not lie in this closure and consequently the desuspension $\Sigma^{-1}=\bbZ[-1]\otimes-$ cannot be expressed via conical homotopy colimits.

\section{A theorem of Dwyer and Kan}

The goal of this section is to study $\mcV$-functors between small $\mcV$-categories which induce Quillen equivalences on the corresponding enriched presheaf categories.

Let $F\col\mcD\ra\mcC$ be a $\mcV$-functor between small locally $\mcV$-flat $\mcV$-categories. There is then a Quillen adjunction
\[F_!\col[\mcD^\op,\mcV]\rightleftarrows[\mcC^\op,\mcV]\loc F^*.\]
Both these functors preserve weighted colimits and on representables $F_!\mcD^d=\mcC^{Fd}$ so that
\[F_!W=Wd\star_{d\in\mcD}\mcC^{Fd}=W\star_\mcD y_\mcC F.\]
Of course $F^*V=VF$. Now we come to the homotopy properties of the two functors. Since they both preserve pointwise cofibrant diagrams and weighted colimits it is also true that they preserve homotopy weighted colimits. Moreover the same is true for the unit $\eta$ and the counit $\varepsilon$ of the adjunction $F_!\dashv F^*$. We will describe these explicitly now -- the counit
\[F_!F^*\mcC^c=F_!\mcC(F-,c)=\mcC(Fd,c)\star_{d\in\mcD}\mcC(-,Fd)\xlra{\varepsilon}\mcC(-,c)=\mcC^c\]
is the composition map and the unit
\[\mcD^d=\mcD(-,d)\xlra{\eta}\mcC(F-,Fd)=F^*\mcC^{Fd}=F^*F_!\mcD^d\]
is simply given by the functor $F$. It will be also useful that $F^*$ preserves all weak equivalences.

We say that $F$ is homotopically full and faithful (the name is taken from \cite{GoerssJardine}) if each component $\mcD(d',d)\ra\mcC(Fd',Fd)$ is a weak equivalence.

\begin{theorem}
Let $F$ be a $\mcV$-functor as above that is surjective on objects. Then the adjunction $F_!\dashv F^*$ is a Quillen equivalence if and only if $F$ is homotopically full and faithful.
\end{theorem}

\begin{proof}
We study the conditions under which the derived unit is a weak equivalence. By the previous we have
\[\xymatrix@R=10pt{
\mcD^d \ar[r]^-{\eta} & F^*F_!\mcD^d \ar[r]^-\sim & (F^*)^\fib F_!\mcD^d \\
\mcD(-,d) \ar@{=}[u] \ar[r]_-F & \mcC(F-,Fd) \ar@{=}[u]
}\]
and thus the derived unit being a weak equivalence is equivalent to $F$ being homotopically full and faithful.

In the opposite direction we compute the derived counit at $\mcC^{Fd}$. Technically we should compute it at its fibrant replacement but since $F^*$ preserves all weak equivalences this is not necessary. By the above the cofibrant replacement of $F^*\mcC^{Fd}=\mcC(F-,Fd)$ may be taken to be $\mcD^d=\mcD(-,d)$ with the augmentation $\mcD^d\xra{F}F^*\mcC^{Fd}$ so that
\[\xymatrix@R=10pt{
(F_!)^\cof F^*\mcC^{Fd} \ar[r] \ar@{=}[d] & F_!F^*\mcC^{Fd} \ar[rr]^-\varepsilon \ar@{=}[d] & & \mcC^{Fd} \ar@{=}[d] \\
\mcD(d',d)\star_{d'\in\mcD}\mcC(-,Fd') \ar[r]_-{F\star\id} & \mcC(Fd',Fd)\star_{d'\in\mcD}\mcC(-,Fd') \ar[rr]_-{\mathrm{composition}} & & \mcC(-,Fd)
}\]
with the bottom row the Yoneda isomorphism (with inverse $f\mapsto\id_d\otimes f$).
\end{proof}

Now we come to the general case when $F$ is not necessarily surjective on objects. We say that an object $c\in\mcC$ is a \emph{homotopy absolute homotopy weighted colimit} of a $\mcV$-functor $F\col\mcD\ra\mcC$ weighted by a cofibrant weight $W\col\mcD^\op\ra\mcV$ if the homotopy weighted colimit $W\star_\mcD y_\mcC F$ in $[\mcC^\op,\mcV]$ is weakly equivalent to $\mcC^c$.

In the non-homotopy case absolute colimits may be characterized precisely as those colimits which are preserved by the Yoneda embedding, i.e.~$W\star_\mcD F$ is absolute if and only if
\[W\star_\mcD y_\mcC F\xlra\cong y_\mcC(W\star_\mcD F)\]
(we will prove this in the homotopy setup in the next paragraph). We defined homotopy absolute homotopy weighted colimits as homotopy versions of this characterization since in $\mcC$ there might not be enough objects to do homotopy theory and also to make sense of $W\star_\mcD F$.

When $\mcC$ is locally cofibrant and fibrant our definition may be rephrased as follows. There is a pointwise weak equivalence $W\star_\mcD y_\mcC F\xlra{\sim}\mcC^c$ which by adjunction corresponds to a morphism
\[W\ra[\mcC^\op,\mcV](y_\mcC F-,\mcC^c)\cong\mcC(F-,c)\]
in $[\mcD^\op,\mcV]$. Thus there is a $W$-cone from $F$ to $c$ in $\mcC$ which is a ``weighted colimit cone up to homotopy'' and which is moreover preserved by the Yoneda embedding. To see that it is then in fact preserved by any pointwise cofibrant functor $D\col\mcC\ra\mcM$ to any model $\mcV$-category $M$ observe that $D$ extends uniquely (up to isomorphism!) to a cocontinuous $\mcV$-functor $[\mcC^\op,\mcV]\ra\mcM$ sending $V\mapsto V\star_\mcC D$ and it is moreover left Quillen. Therefore it preserves homotopy colimits (up to weak equivalence) and if $\mcC^c$ is a homotopy weighted colimit $W\star_\mcD y_\mcC F$ then also $Dc$ will be a homotopy weighted colimit $W\star_\mcD DF$.

One may do a lot of homotopy theory in locally fibrant $\mcV$-categories, see \cite{RosickyEnriched}.

\begin{theorem}\label{theorem_Dwyer_Kan_absolute}
Let $F$ be a $\mcV$-functor as above. Then the adjunction $F_!\dashv F^*$ is a Quillen equivalence if and only if $F$ is homotopically full and faithful and every object $c\in\mcC$ is a homotopy absolute homotopy weighted colimit of $F$.
\end{theorem}

\begin{proof}
We have seen that $\bbL F_!\col\Ho[\mcD^\op,\mcV]\ra\Ho[\mcC^\op,\mcV]$ is fully faithful. We study now when it is moreover essentially surjective, i.e.~when every object of $[\mcC^\op,\mcV]$ is weakly equivalent to the image under $F_!$ of some cofibrant weight. In particular it is necessary that
\[\mcC^c\simeq F_!W\cong W\star_\mcD F_!y_\mcD\cong W\star_\mcD y_\mcC F\]
for some cofibrant weight $W\in[\mcD^\op,\mcV]$.

In the opposite direction observe that the derived counit $(F_!)^\cof F^*\Ra\Id$ is a weak equivalence on each $\mcC^{Fd}$ (again we do not need its fibrant replacement). If $\mcC^c$ is (weakly equivalent to) a homotopy weighted colimit of $y_\mcC F$, i.e.~to a homotopy weighted colimit of a diagram of representables of the form $\mcC^{Fd}$, then the derived counit at $\mcC^c$ will be a weak equivalence too since both $(F_!)^\cof$ and $F^*$ preserve homotopy weighted colimits. Since every object is a homotopy weighted colimit of representables the derived counit will be a weak equivalence on the whole of $\Ho[\mcC^\op,\mcV]$.
\end{proof}

Let us describe the derived counit under the conditions from the theorem at the representable presheaf $\mcC^c$. Since the derived counit is a weak equivalence this will give us in return more information about homotopy absolute homotopy weighted colimits. We denote the cofibrant replacement functor by $Q$.
\[\xymatrix@R=10pt{
& (F_!)^\cof F^*\mcC^c \ar[r] \ar@{=}[d] & F_!F^*\mcC^c \ar[r]^-\varepsilon \ar@{=}[d] & \mcC^c \ar@{=}[d] \\
F^*\mcC^c\mathbin{\widehat\star_\mcD}y_\mcC F \ar@{=}[r] & QF^*\mcC^c\star_\mcD y_\mcC F \ar[r] & F^*\mcC^c\star_\mcD y_\mcC F \ar[r] & \mcC^c
}\]
This says that $\mcC^c$ is in fact canonically weakly equivalent to a derived weighted colimit of $y_\mcC F$ namely that with the weight $F^*\mcC^c$ (or a homotopy weighted colimit with weight the cofibrant replacement thereof). Evaluating at $c'$ the natural map (given by composition)
\[\mcC(Fd,c)\mathbin{\widehat\star_{d\in\mcD}}\mcC(c',Fd)\xlra{\sim}\mcC(c',c)\]
must be a weak equivalence in order for $c$ to be a homotopy absolute homotopy weighted colimit of $F$. Moreover it is clear that $F$ satisfies the conditions of the previous theorem if and only if $F^\op$ does (which should be compared with the well known fact (\cite{KellySchmitt}, Proposition~7.3) that absolute weighted colimits are those which may be at the same time expressed as (absolute) weighted limits).

It is possible to describe this situation in more details for specific examples. We start with the case $\mcV=\sSet$. We say that $c\in\mcC$ is a homotopy retract of an object $c'\in\mcC$ if there exist morphisms $c\xra{f}c'\xra{g}c$ whose composition is homotopic to identity (which simply means that the two vertices of $\mcC(c,c)$ can be joined by a sequence of edges). This is also equivalent to $c$ being a retract of $c'$ in $\pi_0\mcC$.

The following theorem (attributed to Dwyer and Kan) can be found in \cite{GoerssJardine}, Theorem VIII.2.13.

\begin{theorem}
Let $F\col\mcD\ra\mcC$ be a simplicial functor between small simplicial categories. Then the adjunction $F_!\dashv F^*$ is a Quillen equivalence if and only if $F$ is homotopically full and faithful and every object $c\in\mcC$ is a homotopy retract of an object in the image of $F$.
\end{theorem}

\begin{proof}
Assuming that $F_!\dashv F^*$ is a Quillen equivalence we apply $\pi_0$ to the weak equivalence
\[\mcC(Fd,c)\mathbin{\widehat\star_{d\in\mcD}}\mcC(c',Fd)\xlra{\sim}\mcC(c',c)\]
and obtain an isomorphism
\[\pi_0\mcC(Fd,c)\times_{d\in\pi_0\mcD}\pi_0\mcC(c',Fd)\xlra{\cong}\pi_0\mcC(c',c).\]
Specializing to $c'=c$ the preimage of the identity comes from a pair of maps $c\xra{i}Fd\xra{r}c$ whose composition $ri\in\mcC(c,c)$ is homotopic to identity. In effect $c$ is a homotopy retract of $Fd$.

If on the other hand $c$ is a homotopy retract of $Fd$ then also $\mcC^c$ is a homotopy retract of $\mcC^{Fd}$ therefore the derived counit at $\mcC^c$ is a weak equivalence as a homotopy retract of such.
\end{proof}

The same proof works for $\Cat$ with $\pi_0$ replaced by the set of isomorphism classes of objects, i.e.~1-morphisms. We obtain the following characterization.

\begin{theorem}
Let $F\col\mcD\ra\mcC$ be a 2-functor between small 2-categories. Then the adjunction $F_!\dashv F^*$ is a Quillen equivalence if and only if $F$ is full and faithful on hom-objects and every object $c\in\mcC$ is a homotopy retract of an object in the image of the 2-functor $F$.
\end{theorem}

A similar characterization can be made for $\mcV=\sAb$ and $\mcV=\Ch_+$. Since both categories are enriched over $\Ab$ we have the biproduct which, being defined equationally, is absolute. In terms of colimits this means that finite coproducts are absolute and it follows that they are also homotopy absolute.

\begin{theorem}
Let $\mcV$ be one of $\sAb$, $\Ch_+$. Let $F$ be a $\mcV$-functor between small locally $\mcV$-flat $\mcV$-categories. Then the adjunction $F_!\dashv F^*$ is a Quillen equivalence if and only if $F$ is homotopically full and faithful and every object $c\in\mcC$ is a homotopy retract of a finite coproduct of objects in the image of $F$.
\end{theorem}

\begin{proof}
The proof is exactly the same with $\pi_0$ replaced by $H_0$ so that
\[H_0\mcC(Fd,c)\otimes_{d\in H_0\mcD}H_0\mcC(c',Fd)\xlra{\cong}H_0\mcC(c',c)\]
yields maps $c\xra{i_k}Fd_k\xra{r_k}c$, $i=1,\ldots,n$, for which $r_1i_1+\cdots+r_ni_n$ is homotopic to $\id_c$. This means that $\mcC^c$ is a homotopy retract of $\mcC^{Fd_1}\oplus\cdots\oplus\mcC^{Fd_n}$.

In the opposite direction the derived counit at $\mcC^{Fd_1}\oplus\cdots\oplus\mcC^{Fd_n}$ is a weak equivalence and thus also at any homotopy retract of it.
\end{proof}

There are more homotopy absolute homotopy weighted colimits over $\Ch$. Let $\mcC$ be the full dg-subcategory of $\Ch$ on the unit $\bbZ$ and its suspension $\bbZ[1]$. Observe that for any object $M\in\mcM$ in a dg-category $\mcM$ it holds $M[1]=\bbZ[1]\otimes M\cong\hom(\bbZ[-1],M)$ whenever one of these exists. Therefore in $[\mcC^\op,\Ch]$ one has
\[\mcC(-,\bbZ[1])\cong\mcC(-,\bbZ)[1]\]
and in particular $\mcC^{\bbZ[1]}$ is a homotopy weighted colimit of a diagram with values only $\mcC^\bbZ$ while $\bbZ[1]$ is certainly not a homotopy retract of a coproduct of $\bbZ$'s. In particular if $\mcD$ is the full subcategory of $\mcC$ on the object $\bbZ$ then the inclusion $\mcD\hra\mcC$ induces a Quillen equivalence
\[[\mcC^\op,\Ch]\xlra{\sim}[\mcD^\op,\Ch].\]
A similar situation occurs for mapping cones (i.e.~homotopy cokernels) and it should follow from results of section~6 of \cite{Keller} -- namely part a) of the first lemma -- that the homotopy absolute homotopy weighted colimits are generated by (de)suspensions, mapping cones (which give together finite coproducts) and homotopy retracts.

A generalization of Theorem~\ref{theorem_Dwyer_Kan_absolute} to the case of diagrams in a cofibrantly generated model $\mcV$-category $\mcM$ is the following.

\begin{theorem}
Let $F\col\mcD\ra\mcC$ be a $\mcV$-functor between small $\mcM$-flat $\mcV$-categories. Then the adjunction
\[F_!\col[\mcD^\op,\mcM]\rightleftarrows[\mcC^\op,\mcM]\loc F^*\]
is a Quillen equivalence if $F$ is homotopically full and faithful and every object $c\in\mcC$ is a homotopy absolute homotopy weighted colimit of $F$.
\end{theorem}

\begin{proof}
Both $F_!$ and $F^*$ are given by weighted colimits with respective weights $F_!\mcD^\bullet$ and $F^*\mcC^\bullet$ and thus $(F_!)^\cof$ by a derived weighted colimit or alternatively on pointwise cofibrant diagrams as a homotopy weighted colimit (with the weight the cofibrant replacement of the one mentioned). This is so regardless of the category $\mcM$ and the same happens for the derived unit and counit which are given by transformations of weights. Since one gets a Quillen equivalence for $\mcM=\mcV$ by the previous these tranformations must be weak equivalences and thus one gets a Quillen equivalence for arbitrary $\mcM$.
\end{proof}

\appendix
\section{Extra degenerate augmented simplicial objects}

A simplicial object in $\mcM$ is a functor $X\col\Delta^\op\ra\mcM$. We say that it is \emph{augmented} if there is given a map $X_0\ra X_{-1}$ which coequalizes $d_0,d_1\col X_1\ra X_0$. This data can be organized into a natural transformation $X\ra\const_{\Delta^\op}X_{-1}$ from $X$ to the constant simplicial object at $X_{-1}$. More conveniently for us it may be also organized into a functor $\Delta_\varepsilon^\op\ra\mcM$ from the opposite of the category $\Delta_\varepsilon$ of all finite ordinals with the empty ordinal denoted by $[-1]$. When $X$ is augmented by $X_{-1}$ as above we write for simplicity $X\ra X_{-1}$.

We say that the augmented simplicial object $X\ra X_{-1}$ \emph{has extra degeneracy} if there are given $s_{-1}\col X_{n-1}\ra X_n$ for all $n\geq 0$ satisfying certain relations (see III.5 of \cite{GoerssJardine}). These are best formulated by describing the structure of an extra degenerate augmented simplicial object again as a contravariant functor $\widetilde X\col\Delta_\bot^\op\ra\mcM$. The category $\Delta_\bot$ has as objects non-empty finite ordinals but we prefer to write them in the following way $[n]_*=\{-1,0,\ldots,n\}$ so that $X_n=\widetilde X[n]_*$. Morphisms are all monotone maps which preserve the bottom element $-1$. In this notation the extra degeneracy map $s_{-1}$ comes from a morphism $s^{-1}\col[n]_*\ra[n-1]_*$ that repeats $-1$ twice and is otherwise injective. We have a canonical embedding $\Delta_\varepsilon\ra\Delta_\bot$ sending $[n]$ to $[n]_*$ and restriction along this embedding associates to $\widetilde X$ the ``underlying augmented simplicial object''.

\begin{proposition}\label{prop_extra_degenerate}
Let $\widetilde X\col\Delta_\bot^\op\ra\mcM$ be pointwise cofibrant and denote its underlying augmented simplicial object by $X\ra X_{-1}$. Then the canonical map
\[\hocolim_{\Delta^\op}X\ra X_{-1}\]
is a weak equivalence.
\end{proposition}

\begin{proof}
Consider the following diagram in $\Cat$.
\[\xymatrix{
\Delta \ar@{c->}[r]^-J & \Delta_\bot & [-1]_*. \ar@{d->}[l]
}\]
We claim that both functors are homotopy left cofinal (see Definition~19.6.1 of \cite{Hirschhorn}). The $[-1]_*$ stands for the full subcategory of $\Delta_\bot$ on this object and has only the identity morphism.

We will first show how this yields a proof of the proposition. The functors on the opposite categories are homotopy right cofinal and therefore induce weak equivalences on homotopy colimits,
\[\hocolim\nolimits_{\Delta^\op}X\xlra{\sim}\hocolim\nolimits_{\Delta_\bot^\op}\widetilde X\xlla{\sim}X_{-1}.\]
To relate this cospan to the canonical map from the statement observe that there is also a unique functor $\Delta\ra[-1]_*$ fitting into
\[\xymatrix@C=10pt{
\Delta \ar[rr] \ar[dr] & & [-1]_* \ar[dl] & & \hocolim_{\Delta^\op}X \ar[rr] \ar[dr]_-{\sim} & & X_{-1} \ar[dl]^-{\sim} \\
& \Delta_\bot \POS[];[u]**{}?(.6)*{\scriptstyle\Ra} & & & & \hocolim_{\Delta_\bot^\op}\widetilde X \POS[];[u]**{}?(.6)*{\scriptstyle\Ra}
}\]
where the arrow $\Ra$ denotes a natural transformation in the first triangle while a left homotopy in the second triangle. Thus the top arrow -- the map from the statement -- has to be a weak equivalence too.

It remains to prove homotopy left cofinality of the above functors. The case of $[-1]_*$ is implied by the fact that $[-1]_*$ is the initial object of $\Delta_\bot$. Now the functor $J$ is a left adjoint
\[\Delta_\bot(J[m],[n]_*)\cong\Delta([m],[n+1])\]
hence $J\da[n]_*\cong\Delta\da[n+1]$ has a terminal object and so $J$ is also homotopy left cofinal.
\end{proof}

\begin{remarks}\hfill
\begin{enumerate}
\item
There is an obvious variation: if the underlying augmented simplicial object $X\ra X_{-1}$ is Reedy cofibrant (which simply means $X$ Reedy cofibrant and $X_{-1}$ cofibrant) then the canonical map $|X|\ra X_{-1}$ is a weak equivalence (of cofibrant objects).

\item
The proposition implies that homotopy colimits of simplicial objects which admit extra degeneracy are homotopy absolute. This means that for any functor $F$ preserving cofibrant objects and weak equivalences between them the canonical map
\[\hocolim\nolimits_{\Delta^\op}FX\lra F\hocolim\nolimits_{\Delta^\op}X\]
is a weak equivalence whenever the simplicial object $X$ admits an extra degeneracy.

\item
Alternatively one may require extra degeneracy maps $s_n\col X_{n-1}\ra X_n$. This leads to the category $\Delta_\top$ with extra top elements adjoined and asked to be preserved by the morphisms. Everything works exactly the same and in fact $\Delta_\top\cong\Delta_\bot$.
\end{enumerate}
\end{remarks}

\section{Cubical diagrams and their pushout corner maps}

For a finite set $S$ we consider the category $\mcP(S)$ of all subsets of $S$ and the category $\mcP_0(S)$ of all proper subsets. We call any diagram $X\col\mcP(S)\ra\mcM$ an \emph{$S$-cubical diagram}. We denote the Kan extension of its restriction to $\mcP_0(S)$ by $X_0$. Since $\mcP_0(S)$ is full $X_0(I)=X(I)$ for $I\in\mcP_0(S)$ and $X_0(S)=\colim_{\mcP_0(S)}X|_{\mcP_0(S)}$. The pushout corner map of $X$ is the canonical map
\[\pcm X\col X_0(S)\lra X(S).\]

Suppose that $S=S_1\sqcup S_2$. We may then compute the pushout corner map for $X$ by considering the pushout corner maps in the two ``directions'' $S_1$ and $S_2$. Concretely observe that
\[\mcP(S_1\sqcup S_2)\cong\mcP(S_1)\times\mcP(S_2)\]
and consider the left Kan extension of the restriction to $\mcP_0(S_1)\times\mcP(S_2)$ and denote it by $X_1$ and similarly the left Kan extension of the restriction to $\mcP(S_1)\times\mcP_0(S_2)$ will be denoted by $X_2$ and the one for $\mcP_0(S_1)\times\mcP_0(S_2)$ by $X_{12}=(X_1)_2$. The first two may be computed by fixing one variable and considering $X$ as a functor of the other variable and computing the pushout corners in that variable. We obtain a square
\[\xymatrix{
X_{12}(S) \ar[r] \ar[d] & X_1(S) \ar[d] \\
X_2(S) \ar[r] & X(S)
}\]
We will see that its pushout corner map is canonically isomorphic to $\pcm X$. In fact we have $X_1(S)=\colim X|_{\mcP_0(S_1)\times\mcP(S_2)}$ and similarly for the remaining corners. The pushout is therefore isomorphic to the colimit of the restriction to the union
\[(\mcP_0(S_1)\times\mcP(S_2))\sqcup_{(\mcP_0(S_1)\times\mcP_0(S_2))}(\mcP(S_1)\times\mcP_0(S_2))=\mcP_0(S)\]
and finally the pushout corner map is $X_0(S)\ra X(S)$ as claimed.

If $X$ is an $S$-cubical diagram in $\mcM$, $Y$ is a $T$-cubical diagram in $\mcN$ and there is given a bifunctor $-\odot-\col\mcM\otimes\mcN\ra\mcQ$ then we denote by $X\odot Y$ the resulting ($S\sqcup T$)-cubical diagram in $\mcQ$ which has, for $I\in\mcP(S)$ and $J\in\mcP(T)$,
\[(X\odot Y)(I\sqcup J)=X(I)\odot Y(J).\]
Under the assumption that the bifunctor $-\odot-$ is cocontinuous in each variable the above square becomes
\[\xymatrix@C=50pt{
X_0(S)\odot Y_0(T) \ar[r]^{\id\odot(\pcm Y)} \ar[d]_{(\pcm X)\odot\id} & X_0(S)\odot Y(T) \ar[d]^{(\pcm X)\odot\id} \\
X(S)\odot Y_0(T) \ar[r]_{\id\odot(\pcm Y)} & X(S)\odot Y(T)
}\]
and we may conclude
\[\pcm(X\odot Y)\cong\pcm((\pcm X)\odot(\pcm Y)).\]

A cubical diagram is called \emph{cofibrant} if all the pushout corner maps of all its faces are cofibrations. In particular for 0-faces this means that all the objects appearing in the diagram are cofibrant, for 1-faces this means that all the maps appearing in the diagram are cofibrations. These are precisely the cofibrant objects in the Reedy=projective structure on the cubical diagrams.

Suppose now that in addition to $-\odot-$ being cocontinuous in each variable it is also a left Quillen bifunctor. Let both $X$ and $Y$ be cofibrant cubical diagrams and consider $X\odot Y$. Since its faces are the $\odot$-products of faces of $X$ and $Y$ we may compute the respective pushout corner maps from the above square. In particular they are all cofibrations.

We say that a cofibrant cubical diagram is \emph{homotopy cocartesian} if its pushout corner map is a weak equivalence (and thus a trivial cofibration). By the very same argument the $\odot$-product of two cofibrant cubical diagrams one of which is homotopy cocartesian is again homotopy cocartesian.

\begin{proof}[Proof of Proposition~\ref{prop_Reedy_cofibrancy}]
We need to compute the latching objects of $\sfB F$ and their maps into $\sfB F$. These are clearly certain coproducts
\[L_n\sfB F=\coprod_{c_0,\ldots,c_n}\mcC_\mathrm{deg}(c_0,\ldots,c_n)\otimes F(c_n,c_0)\]
which we will describe now. Consider the 1-cubes in $\mcV$
\[X_i=\left\{\begin{array}{ll}
0\ra\mcC(c_{i-1},c_i) & \textrm{ if }c_{i-1}\neq c_i \\
\sfI\ra\mcC(c_{i-1},c_i) & \textrm{ if }c_{i-1}=c_i
\end{array}\right.\]
and $Y$ which is simply $0\ra F(c_n,c_0)$. They clearly depend on $c_0,\ldots,c_n$ but we will not reflect this in the notation. The map in question $L_n\sfB F\ra\sfB_nF$ is easily isomorphic to
\[\coprod_{c_0,\ldots,c_n}\pcm(X_1\otimes\cdots\otimes X_n\otimes Y).\]
Since each $X_i$ is $\mcM$-flat, $Y$ is cofibrant and $-\otimes-\col\mcV\otimes\mcM\ra\mcM$ is ``left Quillen'' with respect to $\mcM$-flat maps in $\mcV$ and cofibrations in $\mcM$ the claim is proved.
\end{proof}

\section{Removing tensors with simplices}

Let again $\mcV=\sSet$ and $\mcC\subseteq\mcM^\cof$ a full subcategory of a simplicial model category $\mcM$ which is homotopy dense. We saw in this situation that $\Delta\ltimes\mcC$ is conical homotopy dense and that in general $\mcC$ is not and needs to be enlarged. We will show that if $\mcC$ is closed under tensor products with all standard simplices $\Delta^k$ we do not need to add them again. This result may be applied to presheaves over $\sAb$ to replace the non-full subcategory $\Delta\ltimes\mcF$ by a (bigger) full one -- the closure of $\mcF$ under tensors with simplices.

\begin{lemma} The following statements hold.
\begin{enumerate}
\item[(a)] If $\mcC$ is closed in $\mcM$ under taking tensors with simplices then $\mcC$ is conical homotopy dense.
\item[(b)] Let $\mcD$ denote the closure of $\mcC$ in $\mcM$ under taking tensors with simplices. Then $\mcD$ is still homotopy dense and by the previous point also conical homotopy dense.
\end{enumerate}
\end{lemma}

\begin{proof}
We express the first statement as the claim that
\[\hocolim_{\mcC_0\da M}JP_M\lra\hocolim_{(\Delta\ltimes\mcC)_0\da M}\widetilde J\widetilde P_M\]
is a weak equivalence where $P_M\col\mcC_0\da M\ra\mcC_0$ is the projection functor, $J\col\mcC\ra\mcM$ is the embedding while $\widetilde P_M$ and $\widetilde J$ are the variants for $(\Delta\ltimes\mcC)_0$. Observe that there is an adjunction
\[\xymatrix{
(\Delta\ltimes\mcC)_0 \ar@<2pt>[r] & \mcC_0 \ar@<2pt>[l]
}\]
whose left adjoint is $(k,c)\mapsto\Delta^k\otimes c$ and the right adjoint $c\mapsto(0,c)$. The counit of the adjunction is an isomorphism while the unit is
\[(k,c)\lra(0,\Delta^k\otimes c)\]
given by the universal $\Delta^k\ra\mcC(c,\Delta^k\otimes c)$. Since the adjuntion is over $\mcM$ (via the functors $J$ and its extension $\widetilde J$) it induces
\[\xymatrix{
(\Delta\ltimes\mcC)_0\da M \ar@<2pt>[r] & \mcC_0\da M \ar@<2pt>[l]
}\]
which is still over $\mcM$ this time via functors $\widetilde J\widetilde P_M$ and $JP_M$. Therefore this adjunction induces a homotopy equivalence of the homotopy colimits (functors give the maps and natural tranformations the homotopies).

We compare now the homotopy density of $\mcC$ and $\mcD$ via weighted homotopy colimits where we denote the various embeddings as in the diagram
\[J\col\mcC\xlra{I}\mcD\xlra{J'}\mcM.\]
The derived counit for $\mcC$ is a weak equivalence
\[M\xlla{\sim}\mcM(J-,M)\mathbin{\widehat{\star}_\mcC}J=I^*\mcM(J'-,M)\star_\mcC J^\cof\cong\mcM(J'-,M)\star_\mcD I_!J^\cof.\]
We claim now that the natural transformation $I_!J^\cof\ra J'$ adjoint to the cofibrant replacement $J^\cof\ra J=I^*J'$ is a weak equivalence so that $I_!J^\cof$ may be taken as $(J')^\cof$ and the above derived counit for $\mcC$ is in fact the same as that for $\mcD$. Since $I$ is full the natural transformation $J^\cof\ra I^*I_!J^\cof$ is an isomorphism and in the diagram
\[J^\cof\xlra\cong I^*I_!J^\cof\lra I^*J'=J\]
(whose composition is the augmentation $J^\cof\xra\sim J$) the second map must be a weak equivalence. This means that $I_!J^\cof\ra J'$ is a weak equivalence on the image of $I$. But every object of $\mcD$ is simplicially homotopy equivalent to an object in $\mcC$. Since the functors and transformations are simplicial the claim is proved.
\end{proof}

Together with Lemma~\ref{lem_conical_dense_implies_simplicial} we proved 
\begin{scorollary}
Suppose that a full subcategory $\mcC\subseteq\mcM^\cof$ of a simplicial model category $\mcM$ is closed under tensors with simplices. Then the following are equivalent
\begin{enumerate}
\item[(a)] $\mcC$ is conical homotopy dense,
\item[(b)] $\mcC$ is simplicially homotopy dense.
\end{enumerate}
\end{scorollary}


\begin{thebibliography}{Ro2}

\bibitem[Ba]{Badzioch}
Badzioch, B., Algebraic theories in homotopy theory, \emph{Ann. of Math. (2)} \textbf{155} (2002), no. 3, 895--913.

\bibitem[Be]{Bergner}
Bergner, J. E., Rigidification of algebras over multi-sorted theories, \emph{Algebr. Geom. Topol.} \textbf{6} (2006), 1925–-1955.

\bibitem[CP]{CordierPorter}
Cordier, J., Porter, T., Homotopy coherent category theory, \emph{Trans. Amer. Math. Soc.} \textbf{349} (1997), no. 1, 1–-54.

\bibitem[Du]{Dugger}
Dugger, D., Universal homotopy theories, \emph{Adv. Math.} \textbf{164} (2001), no. 1, 144–-176.

\bibitem[GJ]{GoerssJardine}
Goerss, P. G., Jardine, J. F., \emph{Simplicial homotopy theory}, Birkhäuser Verlag, Basel, 1999

\bibitem[Hi]{Hirschhorn}
Hirschhorn, P. S., \emph{Model categories and their localizations}, American Mathematical Society, Providence, RI, 2003

\bibitem[Ho]{Hovey}
Hovey, M., \emph{Model categories}, American Mathematical Society, Providence, RI, 1999

\bibitem[Ke${}^\mathrm{r}$]{Keller}
Keller, B., Deriving DG categories, \emph{Ann. Sci. École Norm. Sup. (4)} \textbf{27} (1994), no. 1, 63–-102.

\bibitem[Ke${}^\mathrm{y}$]{Kelly}
Kelly, G. M., \emph{Basic concepts of enriched category theory}, Repr. Theory Appl. Categ. No. 10 (2005).

\bibitem[KS]{KellySchmitt}
Kelly, G. M., Schmitt, V., Notes on enriched categories with colimits of some class, \emph{Theory Appl. Categ.} \textbf{14} (2005), no. 17, 399–-423.

\bibitem[LR]{RosickyEnriched}
Lack, S., Rosick\'y, J., Homotopy theory of enriched categories, in preparation

\bibitem[Ro]{RosickySifted}
Rosick\'y, J., On homotopy varieties, \emph{Adv. Math.} \textbf{214} (2007), no. 2, 525–-550.

\bibitem[St]{Stephan}
Stephan, M., Elmendorf's theorem for cofibrantly generated model categories, Master thesis, 2010

\bibitem[We]{Weibel}
Weibel, C., Homotopy ends and Thomason model categories, \emph{Selecta Math. (N.S.)} \textbf{7} (2001), no. 4, 533–-564.

\end{thebibliography}
\end{document}